\documentclass[11pt,a4paper]{article}
\usepackage{amscd}
\usepackage{enumerate}

\usepackage{amsmath, amssymb, latexsym}
\usepackage{amsfonts}
\usepackage{amsthm}
\usepackage{relsize}
\usepackage{setspace}
\usepackage{geometry}
\usepackage{url}
\usepackage{xspace}
\usepackage{tocloft}
\usepackage{graphics}
\usepackage{graphicx}
\usepackage{lscape}
\usepackage{microtype}
\usepackage{ulem}

\usepackage[usenames, dvipsnames]{color}
\usepackage[utf8]{inputenc}
\usepackage{tikz}

\usepackage[pagebackref=true]{hyperref}
\usepackage[alphabetic]{amsrefs}
\usepackage[english]{babel}

\usepackage{authblk}

\newtheorem{theorem}{Theorem}[section]
\newtheorem{proposition}[theorem]{Proposition}

\newtheorem{corollary}[theorem]{Corollary}
\newtheorem{lemma}[theorem]{Lemma}

\theoremstyle{definition}
\newtheorem{definition}[theorem]{Definition}

\newtheorem{remark}[theorem]{Remark}

\theoremstyle{problem}

\newcommand{\Aut}{\mathrm{Aut}}
\newcommand{\Opp}{\mathrm{Opp}}
\newcommand{\Ker}{\mathrm{Ker}}
\newcommand{\proj}{\mathrm{proj}}
\newcommand{\RR}{\mathbf{R}}
\newcommand{\ZZ}{\mathbf{Z}}

\newcommand{\CC}{\mathbf{C}}
\newcommand{\NN}{\mathbf{N}}

\newcommand{\CAT}{\mathrm{CAT}}
\newcommand{\cat}{$\mathrm{CAT}(0)$\xspace}
\newcommand{\Min}{\mathrm{Min}}
\newcommand{\Ch}{\mathrm{Ch}}
\newcommand{\Iso}{\mathrm{Is}}
\newcommand{\Stab}{\mathrm{Stab}}
\newcommand{\bd}{\partial}

\newcommand{\vareps}{\varepsilon}

\newcommand{\la}{\langle}
\newcommand{\ra}{\rangle}

\def\og{\leavevmode\raise.3ex\hbox{$\scriptscriptstyle\langle\!\langle$~}}
\def\fg{\leavevmode\raise.3ex\hbox{~$\!\scriptscriptstyle\,\rangle\!\rangle$}}

\title
{Gelfand pairs and strong transitivity \\for  Euclidean buildings}

\author[1]{Pierre-Emmanuel Caprace\thanks{F.R.S.-FNRS research associate, supported in part by the ERC (grant \#278469)}}
\author[1]{Corina Ciobotaru\thanks{Supported  by the FRIA}}

\affil[1]{Universit\'e catholique de Louvain, IRMP, Chemin du Cyclotron 2, bte L7.01.02, 1348 Louvain-la-Neuve, Belgique}

\date{First draft: April 22, 2013; revised: September 25, 2013}

\begin{document}

\maketitle

\begin{abstract}
Let $G$ be a locally compact group acting properly, by type-preserving automorphisms on a locally finite thick Euclidean building $\Delta$ and $K$ be the stabilizer of a special vertex  in $\Delta$. It is known that $(G,  K)$ is a Gelfand pair as soon as $G$ acts strongly transitively on $\Delta$; this is in particular the case when $G$ is a semi-simple algebraic group over a local field. We show a converse to this statement, namely: if $(G, K)$ is a Gelfand pair and $G$ acts cocompactly on $\Delta$, then the action is strongly transitive. The proof uses the existence of strongly regular hyperbolic elements in $G$ and their peculiar dynamics on the spherical building at infinity. Other equivalent formulations are also obtained, including the fact that $G$ is strongly transitive on $\Delta$ if and only if it is strongly transitive on the spherical building at infinity. 
\end{abstract}

\renewcommand{\thefootnote}{}


\section{Introduction}

Given a locally compact group $G$ and a compact subgroup $K$ of $G$, we say that $(G, K)$ is a \textbf{Gelfand pair} if the convolution algebra $\mathcal A = C_{c}^{K}(G)$ of compactly supported, continuous, $K$-bi-invariant functions on $G$ is commutative. 
Gelfand pairs play a fundamental  role in the theory of unitary representations of semi-simple real Lie groups, and more generally, of semi-simple algebraic groups over local fields. Moreover, it has been observed that some non-linear locally compact groups are also naturally endowed with a Gelfand pair, which can be used to provide a comprehensive description of  the spherical part of their unitary dual: this is the case of the full automorphism group of a locally finite regular tree (see \cite{Ol} and \cite{FigaNebbia}), as well as some more general groups of tree automorphisms (see \cite{Amann}). On the other hand, the representation theory of most groups  which do not possess a Gelfand pair remains mysterious to a large extent. 

This provides a clear motivation to determine whether a given locally compact group is  endowed with a Gelfand pair. Notice that all classes of groups mentioned above enjoy a common geometric feature: the existence of a proper action on a Riemannian symmetric space or a Euclidean building. In both cases, the induced action on the associated spherical building at infinity is strongly transitive. Therefore, it is natural to ask whether this geometric framework  can be relaxed and still provide examples of Gelfand pairs. Negative results in this direction have been obtained in~\cite{Lec10} and~\cite{APVM11}, where it is proved that, among groups acting strongly transitively on a locally finite thick building of arbitrary type, one has a Gelfand pair only if the building is Euclidean and the compact subgroup is the stabilizer of a special vertex. The main result of the present paper shows moreover that, in the Euclidean case, the hypothesis of strong transitivity cannot be relaxed:

\begin{theorem}
\label{thm:main-thm}
Let $G$ be a locally compact group acting continuously and  properly by type-preserving automorphisms on a locally finite thick Euclidean building $\Delta$.  Then the following are equivalent:
\begin{enumerate}[(i)]
\item
\label{thm:main-thm-i}
$G$ acts strongly transitively on $\Delta$;

\item
\label{thm:main-thm-ii}
$G$ acts strongly transitively on the spherical building at infinity $\partial \Delta$;

\item
\label{thm:main-thm-iii}
$G$ acts cocompactly on $\Delta$,  and $(G,K)$ is a Gelfand pair, where $K$ is the stabilizer of a special vertex.
\end{enumerate} 
\end{theorem}

In addition to those contained in Theorem~\ref{thm:main-thm}, further equivalent characterizations, of independent interest, are presented in Theorem~\ref{transitivity2} below.

We remark that the convolution algebra $\mathcal A= C_{c}^{K}(G)$ is closely related to algebras of operators on the space of functions over the set of vertices of the building $\Delta$. Those algebras can be defined and studied even when $\Delta$ does not admit any strongly transitive automorphism group; this viewpoint has been adopted in a recent work by J.~Parkinson~\cite{Parkinson}.

Notice that the implication from (\ref{thm:main-thm-i}) to (\ref{thm:main-thm-ii}) is easy and well-known to hold in general for abstract groups acting on (not necessarily locally finite) thick Euclidean buildings. For the converse implication `(\ref{thm:main-thm-ii}) $\Rightarrow$ (\ref{thm:main-thm-i})', the major part of our proof, which is given in  Section~\ref{sect:Eucl-strong-trans} below, also holds without assuming that $G$ is locally compact nor that $\Delta$ is locally finite. This hypothesis is only used in the final stage of the proof. However, in Remark~\ref{rem:LocallyInfinite} below, we describe a  concrete strategy towards proving this equivalence in full generality. In the special  case where $\Delta$ is a tree, the equivalence between  (\ref{thm:main-thm-i}) and (\ref{thm:main-thm-ii}) is due to Burger--Mozes \cite[Lemma~3.1.1]{BM}.

The fact that (\ref{thm:main-thm-i}) implies (\ref{thm:main-thm-iii})  is well-known (see \cite[Corollary 4.2.2]{Matsum}). The remaining implication from (\ref{thm:main-thm-iii}) to (\ref{thm:main-thm-i})  is established in Section~\ref{sec:Gelfand-pair}. The special case when $\Delta$ is a tree is mentioned without proof in  the book \cite[page 48]{FigaNebbia}. The proof in the case of Euclidean buildings of arbitrary dimension requires to introduce several new ingredients: further characterizations of the equivalences of (\ref{thm:main-thm-i}) and (\ref{thm:main-thm-ii}), which are given by Theorem~\ref{transitivity2}, as well as the existence in $G$ of automorphisms with a peculiar dynamics, which are called \textbf{strongly regular hyperbolic elements}. We define those as hyperbolic automorphisms of $\Delta$, whose translation axes are contained in a unique apartment, and cross every wall of that apartment. They play a similar role as the regular semi-simple elements in the case of algebraic groups, or as the hyperbolic isometries in the case of tree automorphism groups. The following existence result, established in Section~\ref{sec:strong-elem}, is of independent interest. 

\begin{theorem}\label{thm:ExistenceStronglyReg}
Let $G$ be a group acting cocompactly by automorphisms on a locally finite Euclidean building $\Delta$. Then $G$ contains a strongly regular hyperbolic element. 
\end{theorem}

The special case when $G$ is discrete can be deduced from Ballmann--Brin~\cite{BB}, where they use a Poincar\'e recurrence argument on the geodesic flow.  Our proof is different and is inspired by an argument due to E.~Swenson~\cite{Swe99}.

\medskip
In our considerations, the relevance of strongly regular hyperbolic elements comes from their peculiar dynamical properties, which show some resemblance with the dynamics of hyperbolic isometries of rank one symmetric spaces (or more generally, of Gromov hyperbolic CAT(0) spaces):

\begin{proposition}\label{prop:Dynamics}
Let $\Delta$ be a Euclidean building and $a \in \Aut(\Delta)$ be a type-preserving strongly regular element with unique translation apartment $A$.  Then, for any point  $\xi$ in the visual boundary $ \partial \Delta$, the limit $\lim\limits_{n \to \infty} a^n(\xi)$ exists (in the cone topology) and  belongs to $ \bd A$. In particular, the fixed-point-set of $a$ in $\bd \Delta$ is $\bd A$. 
\end{proposition}

A slightly more precise statement will be established in Proposition~\ref{prop:DynamRegular} below. 

\bigskip
The paper is organized as follows. Section~\ref{sec:strong-elem} is devoted to strongly regular hyperbolic automorphisms. Various properties are described, and the proofs of Theorem~\ref{thm:ExistenceStronglyReg} and Proposition~\ref{prop:Dynamics} are completed.  Section~\ref{sect:Eucl-strong-trans} is devoted to strong transitivity. The goal there is to establish Theorem~\ref{transitivity2}, which contains the equivalence between (\ref{thm:main-thm-i}) and (\ref{thm:main-thm-ii}) from Theorem~\ref{thm:main-thm}, as well as some additional characterizations of independent interest. A good deal of Section~\ref{sect:Eucl-strong-trans} holds in the general context of abstract groups acting   on thick Euclidean buildings. Finally, Section~\ref{sec:Gelfand-pair} is devoted to Gelfand pairs where the proof of Theorem~\ref{thm:main-thm} is  completed. 

\subsection*{Acknowledgement} We thank Jean L\'ecureux for drawing our attention to the references \cite{Parkinson} and \cite{APVM11}. We also thank the referee for his/her comments.

\section{Strongly regular hyperbolic automorphisms}
\label{sec:strong-elem}

Throughout the paper, we view Euclidean buildings both as \cat spaces and as simplicial complexes. We refer to \cite{AB} and \cite{BH99} for the main concepts and basic properties, which will be used freely in the paper.

The goal of this section is to study the basic properties of strongly regular hyperbolic automorphisms of Euclidean buildings and to prove Theorem~\ref{thm:ExistenceStronglyReg}. 

\subsection{Regular and strongly regular isometries}

The following definition is inspired by the notion of $\RR$-hyper-regular elements, appearing in Prasad--Raghunathan's paper \cite{PR72}, or more recently in \cite{BL93}, where they are called $h$-regular elements.

\begin{definition}
Let $X$ be a $\CAT(0)$ space, $\gamma$ be a hyperbolic isometry of $X$ and let 
$$\Min(\gamma):=\{x \in X \; | \; d(x, \gamma(x))=|\gamma| \},$$ 
where $|\gamma|$ denotes the translation length of $\gamma$. We say that $\gamma$ is a \textbf{regular hyperbolic isometry} of $X$ if $\Min(\gamma)$ is at bounded Hausdorff distance from a maximal flat of $X$. 
\end{definition}

It follows that the visual boundary of $\Min(\gamma)$ is a round sphere in $\partial X$. In view of \cite[Prop.~2.1]{Leeb}, this implies that $\Min(\gamma)$ contains a $\gamma$-invariant flat of $X$, which therefore lies at a bounded Hausdorff distance from the corresponding maximal flat. Notice that in a general \cat space, maximal flats need not exist. Moreover, two different maximal flats need not have the same dimension. 

In the case when $X$ is a Riemannian symmetric space or a Bruhat--Tits building, the full isometry group $\Iso(X)$ is known, is very big and it acts strongly transitively on the spherical building at infinity of $X$. Using the known properties of $\Iso(X)$ in those cases, it is easy to construct hyperbolic isometries $\gamma \in \Iso(X)$ such that $\Min(\gamma)$ is a maximal flat. In particular, such elements are regular hyperbolic isometries  in the sense above. Moreover, if $X$ is a proper \cat space and $\gamma$ is a rank one element (in the sense of Ballmann~\cite{Ballmann}), then $\gamma$ is also regular hyperbolic; in that case the relevant maximal flat is a rank one geodesic line of $X$ (see \cite[Chapter~3]{Ballmann}). 

In the present paper we focus on the case where $X$ is a Euclidean building. In this case the maximal flats coincide with the apartments. Therefore, if $\gamma$ is a regular hyperbolic isometry, there is a unique apartment $A$ contained in $\Min(\gamma)$. However, even if $A = \Min(\gamma)$, it could be that the $\gamma$-axes (which are pairwise parallel) are \textbf{singular}, i.e., are parallel to some wall of $A$. The case where this does not happen is central to our purposes and we give it a special name. 

\begin{definition}
Let $\Delta$ be a  Euclidean building and $\gamma \in \Aut(\Delta)$. We say that $\gamma$ is  \textbf{strongly regular hyperbolic} if $\gamma$ is regular hyperbolic and if one (and hence all)  of its translation axes crosses all the walls of the unique apartment contained in $\Min(\gamma)$.
\end{definition}

We record the following easy characterization. 

\begin{lemma}
\label{lem:CharSRH}
Let $\Delta$ be a Euclidean building and $\gamma \in \Aut(\Delta)$ be a type-preserving automorphism. Then $\gamma$ is strongly regular hyperbolic if and only if $\gamma$ is a hyperbolic isometry and the two endpoints of one (and hence all) of its translation axes lie in the interior of two opposite chambers of the spherical building at infinity. 

In particular, if $\gamma$ is strongly regular hyperbolic, then $\Min(\gamma)$ is an apartment. 
\end{lemma}

The apartment $\Min(\gamma)$ will henceforth be called the \textbf{translation apartment} of $\gamma$; it is uniquely determined. 

\begin{proof}[Proof of Lemma~\ref{lem:CharSRH}]
The `only if' part is clear. Conversely, if $\gamma$ is hyperbolic and if the endpoints of some $\gamma$-axis, say $\ell$, lie respectively in the interior of two opposite chambers $c_+, c_- \in \Ch(\partial \Delta)$, then the same holds for all $\gamma$-axes. Consider now the unique apartment $A$ of $\Delta$ whose boundary contains $c_+$ and $c_-$. Since $\gamma$ is type-preserving and it fixes the two endpoints of its translation axes, it follows that $\gamma$ fixes $c_-$ and $c_+$ pointwise. Therefore, we conclude that the apartment $A$ is invariant under $\gamma$. In particular, we obtain that $\gamma$ acts trivially on $\partial A$; in other words $\gamma$ acts by translations on $A$. This proves that $A \subseteq \Min(\gamma)$. To prove the converse inclusion, we remark  that any translation axis $\ell'$ of $\gamma$ has the same endpoints as $\ell$, and must thus be entirely contained in $A$. Thus $A = \Min(\gamma)$. Since the endpoints of $\ell$ lie in the interior of the chambers $c_-,c_+$, the line $\ell$ cannot be parallel to any wall of $A$. Therefore, it meets every such wall. This confirms that $\gamma$ is strongly regular hyperbolic.
\end{proof}

Lemma \ref{lem:CharSRH} motivates the following.

\begin{definition}
Let $\Delta$ be a Euclidean building. A geodesic line $\ell$ is called  \textbf{strongly regular} if its endpoints lie in the interior of two opposite chambers of the spherical building at infinity.
\end{definition}

Thus, in this case, the line $\ell$ is contained in a unique apartment $A$ and $\ell$ crosses every wall of $A$. Lemma~\ref{lem:CharSRH} shows that a type-preserving hyperbolic automorphism is strongly regular if and only if its translation axes are strongly regular.

Notice that strongly regular geodesic lines are defined by means of an asymptotic property, namely a property of its endpoints at infinity. 
An easy but crucial observation is the existence of a local criterion to recognize strongly regular lines:

\begin{lemma}\label{lem:LocalCrit:SRL}
Let $\Delta$ be a  Euclidean building and $\ell$ be a geodesic line. If $\ell$ contains two special vertices which are contained, respectively, in the interiors of two opposite sectors in a common apartment of $\Delta$, then $\ell$ is strongly regular. Conversely, if $\ell$ is strongly regular and contains three special vertices, then at least two of them must be contained in the interiors of two opposite sectors of a common apartment. 
\end{lemma}

\begin{proof}
Let $v, v' \in \ell$ be special vertices respectively contained in the interiors of two opposite sectors $s, s'$ of some apartment $A$. Then $A$ contains the geodesic segment $[v, v']$, which cannot be parallel to any wall of $A$. In particular, there exists at least one chamber $c$ which is contained in the intersection of all apartments containing  both  $v$ and $v'$; in particular $c $ is contained in $A$. Let now $B$ be an apartment containing $\ell$. Then the retraction $\rho_{c, A}$ onto $A$ based at $c$ induces an isomorphism of $B$ onto $A$, which fixes the intersection $A \cap B$ pointwise (see Definition~4.38 and Proposition~4.39 in \cite{AB}). Therefore $\rho_{c, A}$ maps $\ell$ to a geodesic line $\ell_A$ of $A$ containing $[v, v']$. By hypothesis, the line $\ell_A$ is strongly regular. Since the restriction of $\rho_{c, A}$ to $B$ induces an isomorphism $B \to A$, we infer that $\ell$ is also strongly regular, as desired.

The converse statement is straightforward. 
\end{proof}

\subsection{Existence of strongly regular elements}

The first step is the special case of thin buildings.

\begin{lemma}
\label{existance_reg_element}
Let $(W,S)$ be a  Euclidean Coxeter system and $A$ be the associated Coxeter complex. Then $W$ contains  strongly regular hyperbolic elements.
\end{lemma}

\begin{proof}
Recall that the translation subgroup of $W$ acts simply transitively on the set of special vertices of a given type of the Coxeter complex $A$. Let $v$ and $v'$ be two special vertices   respectively contained in the interiors of two opposite sectors of $A$. Then the unique geodesic line $\ell$ through $v$ and $v'$ is strongly regular and the unique translation $t$ mapping $v$ to $v'$ is an element of $W$ having $\ell$ as a translation axis. Thus $t$ is strongly regular hyperbolic by Lemma~\ref{lem:CharSRH}. 
\end{proof}

The following basic construction, valid in general \cat spaces, is a key step in proving the existence of strongly regular hyperbolic automorphisms in groups acting on a locally finite Euclidean building $\Delta$. 
In what follows, given a proper \cat space $X$, we denote its visual boundary by $\partial X$ and endow the space  $X \cup \partial X$ with the \textbf{cone topology}, which is compact (see~\cite[Chap.II.8]{BH99} for the basic definitions). The following lemma is technical, but its proof relies on a standard compactness argument.

\begin{lemma}
\label{lem:basic-const}
Let $X$ be a proper \cat space, $G < \mathrm{Is}(X)$ be any group of isometries and $\rho \colon \RR \to X$ be a geodesic map. Assume there is an increasing sequence $\{t_n\}_{n \geq 0}$ of positive real numbers tending to infinity such that $\sup_n d(\rho(t_n), \rho(t_{n+1})) < \infty$ and the set $\{\rho(t_n)\}_{n \geq 0}$ falls into finitely many $G$-orbits, each of which is moreover discrete. 

Then there exist a sequence $\{g_{n}\}_{n \geq 0} $ in $ G$ and  an increasing sequence $\{f(n)\}_{n\geq 0}$ of positive integers such that  for all $m \geq 0$ and $r \in [-m, m]$ the sequence $\{g_n\circ \rho(t_{f(n)}+~r)\}_{n \geq m}$ is constant.
Furthermore, the limit map 
$$\rho' \colon \RR \to \Delta :  r \mapsto  \lim\limits_{n \to \infty} g_n \circ \rho(t_{f(n)} +r)$$
is geodesic. In particular, for $n \geq m$, the intersection $g_n(\rho(\RR)) \cap \rho'(\RR)$  is a geodesic segment of length~$\geq 2m$.

If in addition $X$ is a locally finite Euclidean building and $\rho(\RR)$ is strongly regular, then  $\rho'(\RR)$ is  strongly regular as well.
\end{lemma}

\begin{proof}
Set $x_n = \rho(t_n )$ and $C=\sup_n d(x_n, x_{n+1})$.  Let also $V = \bigcup_n G(x_n)$ denote the union of the $G$-orbits of the points $x_n$. 
Since the sequence $\{x_n\}_{n \geq }$ falls into finitely many $G$-orbits,   we may find a subsequence $\{x_{f(n)}\}_{n \geq 0}$  and a sequence $\{g_{n}\}_{n \geq 0} \subset G$ such that the sequence $\{ g_{n}(x_{f(n)})\}_{n \geq 0} $ is constant. We set $v_0 = g_{n}(x_{f(n)}) \in V$.

Since $X$ is proper and $V$ is discrete by hypothesis,  any ball around $v_0$ contains finitely many points of $V$. On the other hand, for all $n \geq m \geq 0$, the geodesic segment $g_n \circ \rho([t_{f(n)} - m, t_{f(n)} +m])$, which is contained in the ball of radius $2m$ around $v_0$, contains at least $\lfloor \frac{2m} C \rfloor$  points of $V$. Therefore, we may adjust the sequence $\{f(n)\}_n$ to ensure that for all $ m \geq 0$ the set of geodesic segments $\{g_n \circ \rho ([t_{f(n)}-m, t_{f(n)}+m]) \; :\;n \geq m\}$ is exactly one segment. Thus, by construction, the map $\rho' \colon r \mapsto \lim\limits_{n\to \infty}   g_n \circ \rho (t_n +r)$ is well defined and therefore $\rho'(\RR)$ is the limit, when $n \to \infty$, of the sequence of geodesic lines $g_n (\rho(\RR))$, which are converging uniformly on compact sets. Thus $\rho'(\RR)$ is a also a geodesic line. That $\rho'(\RR)$ is strongly regular in the case when $X$ is a locally finite Euclidean building and $\rho(\RR)$ is strongly regular follows from Lemma~\ref{lem:LocalCrit:SRL}.
\end{proof}

We will also need the following subsidiary fact.

\begin{lemma}
\label{reg_hyp_element}
Let $X$ be a $\CAT(0)$ space and $\rho \colon \RR \to X$ be a geodesic line. Let moreover $h$ be an isometry. Suppose there exist $t <  t' \in \RR$ and $c>0$ such that 
$$h \circ \rho(t + \vareps) = \rho(t' + \vareps)$$
for all $\vareps  \in [-c, c]$. Then $h$ is a hyperbolic element admitting a translation axis containing the segment $[\rho(t),\rho(t')]$.
\end{lemma}

\begin{proof}
Set $x = \rho(t)$ and $y = \rho(t')$, so $h(x) = y$. By hypothesis,   the Alexandrov angle $\angle_{y}(x, h(y))=\angle_{y}(x, h^{2}(x))$ is equal to $\pi$. It follows that the three points $x, y $ and $h(y)=h^{2}(x)$ are collinear and so the geodesic segment $[x,h^{2}(x)]$ contains the point $y$. The same argument used inductively shows that  all points of the sequence $\{h^{n}(x)\}_{n \in \ZZ}$ are contained in a common geodesic line, which we denote by $\ell$. Therefore, by this construction the geodesic line $\ell$ is preserved by $h$. Moreover, $x,y \in \ell$ and by hypothesis $h(x) = y \neq x$. Thus, $h$ is a hyperbolic isometry having $\ell$ as a translation axis.
\end{proof}

The following result is a technical relative of Theorem~\ref{thm:ExistenceStronglyReg}, valid for more general \cat spaces.  We state it separately for the sake of future references. Notice that groups acting isometrically on \cat spaces with discrete orbits are plentiful: they include (not necessarily discrete) groups whose action preserves a locally finite simplicial (or cellular) decomposition of the space, or totally disconnected groups acting minimally on a space with the geodesic extension property (as a consequence of \cite[Th.~6.1]{CM09}).

\begin{proposition}
\label{prop:TechSRH}
Let $X$ be a proper \cat space, $G < \mathrm{Is}(X)$ be any group of isometries and $\rho \colon \RR \to X$ be a geodesic map. Assume that there is an increasing sequence $\{t_n\}_{n \geq 0}$ of positive real numbers tending to infinity such that $\sup_n d(\rho(t_n), \rho(t_{n+1})) < \infty$ and that the set $\{\rho(t_n)\}_{n \geq 0}$ falls into finitely many $G$-orbits, each of which is moreover discrete.

Then there is an increasing sequence $\{f(n)\}_{n}$ of positive integers such that, for all $n> m > 0$, there is a hyperbolic isometry $h_{m, n} \in G$ which has a translation axis containing the geodesic segment $[\rho(t_{f(m)}), \rho(t_{f(n)})]$.

If in addition $X$ is a locally finite Euclidean building and the geodesic line $\rho(\RR)$ is strongly regular, then $h_{m, n}$ is a strongly regular hyperbolic automorphism.
\end{proposition}

\begin{proof}
Let $\{g_n\}_{n\geq 0} \subset G$ and $\{f(n)\}_{n \geq 0} \subset \NN$ be the sequences afforded by  Lemma~\ref{lem:basic-const}. Fix $m > 0$. Then by Lemma~\ref{lem:basic-const} we know that for all $r \in [-m, m]$, the sequence $\{g_n \circ \rho(t_{f(n)} + r)\}_{n\geq m}$ is constant. In particular, for all $n>m$, the element $h_{m, n} = g_m^{-1} g_n \in G$ has the property that 
$$h_{m, n} \circ \rho(t_{f(n)} +r) = \rho(t_{f(m)}+r)$$
for all $r \in [-m, m]$. Therefore, Lemma~\ref{reg_hyp_element} ensures that $h_{m, n}$ is hyperbolic and has a translation axis passing through $\rho(t_{f(m)})$ and $ \rho(t_{f(n)})$. In the case when $X$ is a locally finite Euclidean building and the geodesic line $\rho({\RR})$ is strongly regular, this axis must be therefore strongly regular by Lemma~\ref{lem:LocalCrit:SRL}. In view of Lemma~\ref{lem:CharSRH} we conclude that $h_{m, n}$ is strongly regular hyperbolic as well. 
\end{proof}

We are now able to complete the proof of Theorem~\ref{thm:ExistenceStronglyReg} from the introduction. 

\begin{proof}[Proof of Theorem~\ref{thm:ExistenceStronglyReg}]
We first observe that the subgroup of $\Aut(\Delta)$ consisting of type-preserving automorphisms is of finite index, see \cite[Prop.~A.14]{AB}. In particular, $G$ has a finite index subgroup acting by type-preserving automorphisms, whose action remains therefore cocompact. There is thus no loss of generality in assuming that the $G$-action is type-preserving.

Let $A$ be an apartment in $\Delta$. By Lemma~\ref{existance_reg_element}, the Weyl group of $\Delta$ acting on $A$ contains strongly regular hyperbolic elements. Since any point and so also any vertex of $A$ belongs to an axis of a fixed such strongly regular element, it follows that there exists a geodesic map $\rho \colon \RR \to A$ and some $C>0$ such that $\ell = \rho(\RR)$ is strongly regular and that the sequence $\{\rho(nC)\}_{n \geq 0}$ are special vertices of the same type. The set of all special vertices is clearly discrete. Since the $G$-action is cocompact by hypothesis, it follows that $G$ has finitely many orbits of special vertices. Therefore, the desired conclusion now follows from Proposition~\ref{prop:TechSRH}.
\end{proof}

\subsection{Dynamics of strongly regular elements}
\label{subsect:dynamics-reg}

Hyperbolic isometries  of $\CAT(-1)$ spaces (and more generally of Gromov hyperbolic metric spaces) enjoy remarkable dynamical properties: they have a unique attracting point and a unique repelling point in the visual boundary and any other point of the boundary is contracted by the positive powers of the isometry towards the attracting fixed point. It turns out that this property is shared by rank one isometries of $\CAT(0)$ spaces. However, it cannot be expected that such a peculiar dynamical behavior extends to all regular hyperbolic isometries of $\CAT(0)$ spaces, since a hyperbolic isometry $g$ acts trivially on the visual boundary of $\Min(g)$, which is generally not reduced to a pair of boundary points. 

We shall however see that the dynamics of a strongly regular automorphism of a Euclidean building has a contracting property which is reminiscent from (but weaker than) the hyperbolic case mentioned above. 

In the following statement, the symbol $\rho_{A, c}$ denotes the retraction onto an apartment $A$ and based at the ideal chamber $c \in \Ch(\partial A)$, as it is introduced in Section~\ref{subsubsect:Geom-radical}. The following is a strengthening of Proposition~\ref{prop:Dynamics} from the introduction.

\begin{proposition}
\label{prop:DynamRegular}
Let $\Delta$ be a Euclidean building and $a \in \Aut(\Delta)$ be a  type-preserving strongly regular hyperbolic element with unique translation apartment $A$. Let $c_- \in \Ch(\partial A)$ be the unique chamber containing the repelling fixed point of $a$. 

Then, for any $\xi \in \partial \Delta$, the limit $\lim\limits_{n \to \infty} a^n(\xi)$ exists (in the cone topology) and coincides with $\rho_{A, c_-}(\xi) \in \bd A$. In particular, the fixed-point-set of $a$ in $\bd \Delta$ is $\bd A$. 
\end{proposition}

\begin{proof} 
Let $\xi \in \bd \Delta$ and consider an apartment $B$ whose boundary contains $\xi$ and $c_-$. Let $Q$ be a sector pointing to $c_-$ and contained in the intersection $A \cap B$. Let moreover $p \in Q$ be a base point contained in the interior of $Q$. Since $B$ is convex, the geodesic ray $[p, \xi)$ is entirely contained in $B$. If $\xi \in c_-$, then $\xi $ is fixed by $a$ and the desired conclusion follows. We assume therefore that $\xi \not \in c_-$. In particular, the ray $[p, \xi)$ is not entirely contained in $Q$ and hence there is some $q \in Q$ such that $[p, \xi) \cap Q = [p, q]$.

Since $p$ lies in the interior of $Q$, it follows that the segment $[p, q]$ is of positive length. Therefore, in the apartment $A$, the geodesic segment $[p, q]$ can be extended uniquely to a geodesic ray emanating from $p$; in other words, there is a unique boundary point $\eta \in \bd A$ such that the segment $[p, q]$ is contained in the ray $[p, \eta)$. 

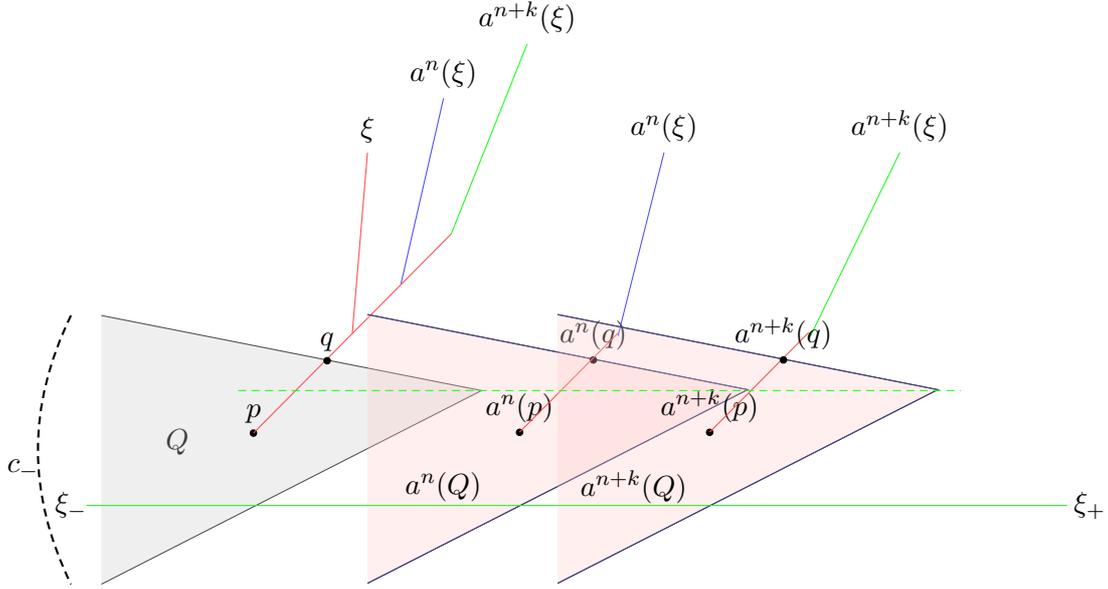
\begin{figure}[t]
\begin{center}
\begin{tikzpicture}
[yscale=1.2,xscale=1, vertex/.style={draw,fill,circle,inner sep=0.3mm},virtual/.style={thick,densely dashed},]



\node[](x) at (-4.50, 0.50) {$Q$};

\filldraw[draw=black!70!white, fill=black!20!white, fill opacity = 0.3] (0:0.75cm) 
(-5.50, 1.90)--(-0.50, 1.07)--(-5.50, -1.07) ;

\path[virtual] (-5.90, 1.90) edge [bend right] (-5.90, -1.07);
\path (-6.53,0.20) node[] {$c_{-}$};
 

\node[vertex]
(p) at (-3.50, 0.60) {};
\path (p) node[above] {$p$};

{
\draw[red!70!white, thin]
 (-3.50, 0.60)-- (-2.20,1.70);
}     

{
\draw[red!70!white, thin]
 (-2.20,1.70)--(-0.90,2.80);
}  

{
\draw[blue!70!white, thin]
 (-1.56,2.25)--(-1.00,4.30);
}  
\path (-1.00,4.30) node[above] {$a^{n}(\xi$)};

{
\draw[green!90!white, thin]
 (-0.90,2.80)--(0.10,4.90);
} 

\path (0.10,4.90) node[above] {$a^{n+k}(\xi$)};

\node[vertex]
(q) at (-2.53, 1.40) {};
\path (q) node[above] {$q$};

{
\draw[red!70!white, thin]
 (-2.20,1.70)--(-2.00, 3.70);
}     

\path (-2.00, 3.70) node[above] {$\xi$};

\path (-5.90,-0.20) node[] {$\xi_{-}$};

\path (7.50,-0.20) node[] {$\xi_{+}$};

\draw[blue,shift={(3.50 cm,0.01 cm)}] (-5.50, 1.90)--(-0.50, 1.07)--(-5.50, -1.07);
\filldraw[draw=black!70!white, fill=red!20!white, fill opacity = 0.3,shift={(3.50cm,0.01 cm)}] (0:0.75cm) 
(-5.50, 1.90)--(-0.50, 1.07)--(-5.50, -1.07) ;

\node[shift={(3.50 cm,0.01 cm)}](x) at (-4.50, -0.00) {$a^{n}(Q)$};
\node[vertex,shift={(3.50 cm,0.01 cm)}]
(p) at (-3.50, 0.60) {};
\path (p) node[above] {$a^{n}(p)$};
\node[vertex,shift={(3.50 cm,0.01 cm)}]
(q) at (-2.53, 1.40) {};
\path (q) node[above] {$a^{n}(q)$};

{
\draw[red!70!white, thin,shift={(3.50 cm,0.01 cm)}]
 (-3.50, 0.60)-- (-2.20,1.70);
}     

\node[shift={(3.50 cm,0.01 cm)}](y) at (-2.24,1.55) {};
{
\draw[blue!70!white, thin]
 (y)--(1.90, 3.70);
}    

\path[] (1.90, 3.70) node[above] {$a^{n}(\xi)$};

\draw[blue,shift={(6cm,0.01 cm)}] (-5.50, 1.90)--(-0.50, 1.07)--(-5.50, -1.07);
\filldraw[draw=black!70!white, fill=red!20!white, fill opacity = 0.3,shift={(6.00cm,0.01 cm)}] (0:0.75cm) 
(-5.50, 1.90)--(-0.50, 1.07)--(-5.50, -1.07) ;

\node[shift={(6.00 cm,0.01 cm)}](x) at (-4.50, -0.00) {$a^{n+k}(Q)$};
\node[vertex,shift={(6.00 cm,0.01 cm)}]
(p) at (-3.50, 0.60) {};
\path (p) node[above] {$a^{n+k}(p)$};

{
\draw[red!70!white, thin,shift={(6 cm,0.01 cm)}]
 (-3.50, 0.60)-- (-2.20,1.70);
}     

\node[shift={(6 cm,0.01 cm)}](z) at (-2.24,1.55) {};
{
\draw[green!90!white, thin]
 (z)--(5.00, 3.70);
}   
  
\path[] (5.00, 3.70) node[above] {$a^{n+k}(\xi)$};

\node[vertex,shift={(6 cm,0.01 cm)}]
(q) at (-2.53, 1.40) {};
\path (q) node[above] {$a^{n+k}(q)$};

{
\draw[green!90!white,thin]
 (-5.70,-0.20)--(7.20,-0.20);
}     

{
\draw[green!90!white,virtual, thin]
 (-3.70,1.07)--(5.80,1.07);
}   

\end{tikzpicture}
\caption{\footnotesize The green line $(\xi_{-},\xi_{+})$ is a translation axis of the strongly regular element $a$; $c_{-}$ is the chamber in $\partial A_{a}$ containing $\xi_{-}$; $Q$ is a sector in $A_{a}$ corresponding to $c_{-}$; the red line between $p$ and $\xi$ is the geodesic ray $[p,\xi)$.}
\end{center}
\end{figure}

We claim that $\lim\limits_{n \to \infty} a^n(\xi) = \eta$. 

Since $a$ is strongly regular, we have $a^n(Q) \supset Q$ for all $n >0$. Moreover $[a^n(p), a^n(q)] = [a^n(p), a^n
(\xi)) \cap a^n(Q)$. Since $a^n$ acts on $A$ as a Euclidean translation, it follows that $[a^n(p), a^n(q)] = a^n([p, q]) $ is parallel to $[p, q]$ (in the Euclidean sense). On the other hand, the apartment $a^n(B)$ contains both the ray $[p, a^n(\xi))$ and the ray $[a^n(p), a^n(\xi))$, which are parallel since they have the same endpoint. It follows that $[p, a^n(\xi))$ contains $[p, q]$. 
We next observe that, since $a$ is strongly regular, for any boundary wall $M$ of $Q$, the distance between $M$ and $a^n(M)$ grows linearly with $n$. This implies that the length of the geodesic segment $[p, a^n(\xi)) \cap a^n(Q)$ tends to infinity with $n$. In particular so does the length of $[p, a^n(\xi)) \cap A$. 

We have thus shown that the geodesic ray $[p, a^n(\xi))$ contains $[p, q]$ for all $n \geq 0$ and also a subsegment of $A$ whose length tends to infinity with $n$. By the definition of $\eta$, this implies that  $[p, a^n(\xi)) \cap [p, \eta) $ is a segment whose length also tends to infinity with $n$. In particular we have $\lim\limits_{n \to \infty} a^n(\xi) = \eta$ and the claim stands proven. 

\medskip
Now, let $\xi' \in \bd B$ be another boundary point of $B$. Since $B$ is a Euclidean flat, the $\CAT(1)$ distance between $\xi$ and $\xi'$ coincides with the angle at $p$ between the rays $[p, \xi)$ and $[p, \xi')$. The claim above implies that this angle also coincides with the $\CAT(1)$ distance between $\lim\limits_{n \to \infty} a^n(\xi) $ and $\lim\limits_{n \to \infty} a^n(\xi')$. Therefore, the restriction to $\bd B$ of the map
$$\bd \Delta \to \bd A : \xi \mapsto \lim\limits_{n \to \infty} a^n(\xi) $$
is an isometry of $\bd B$ onto $\bd A$ fixing $c_-$ pointwise. By definition of the retraction, this implies that $\lim\limits_{n \to \infty} a^n(\xi)  = \rho_{A, c_-}(\xi)$.
\end{proof}

Moreover, Proposition~\ref{prop:DynamRegular} shows that the boundary of the translation apartment (without the `repelling' ideal chamber $c_-$) of a strongly regular hyperbolic element behaves like an attracting locus for the orbits at infinity. The following result highlights an additional `attracting property', with respect to the interior of the building, of the translation apartment of a strongly regular hyperbolic element: geodesic segments joining a point to its image under large powers of the strongly regular element must pass through that apartment. 

\begin{proposition}
\label{lemma_2}
Let $\Delta$ be a  Euclidean building and $a \in \Aut(\Delta)$ be a type-preserving strongly regular hyperbolic element with unique translation apartment $A$. Let also $S$ be a finite set of special vertices and let  $x_{0}  \in S$.

Then for every $T>0$, there exists $N \geq1$ such that for all $n \geq N$ and $x \in S$,  the geodesic segment $[x_{0}, a^{n}(x)]$ contains a subsegment of length greater than $T$ entirely contained in $A$.
\end{proposition}

\begin{proof}
Let $x \in S$. We will find a constant $N_x$ such that for all $n \geq N_x$, the geodesic segment $[x_{0}, a^{n}(x)]$ contains a subsegment of length greater than $T$ entirely contained in $A$. Once this has been done, the constant $N = \max\limits_{x \in S} N_x$ satisfies the desired conclusions. 

Let $\xi_{-},\xi_{+}$ denote the two endpoints of the regular translation axis of $a$. By Lemma~\ref{lem:CharSRH}, these endpoints $\xi_{-}$ and $\xi_{+}$ lie respectively in the interior of some chambers at infinity, say   $c_{-},c_{+}$. In order to find the constant $N_x$, we distinguish several cases. 

Assume first that $x \in A$. Let then $A_{x_{0}}$ be an apartment in $\Delta$ with the property that $x_{0} \in A_{x_{0}}$ and $c_{+}$ is an ideal chamber in $\Ch(\partial A_{x_{0}})$.  As $c_{+}$ is a common ideal chamber of $\Ch(\partial A_{x_{0}})$ and $\Ch(\partial A)$, the apartments $A$ and $A_{x_{0}}$ share a common sector corresponding to $c_{+}$. Denote it by $Q_{x_{0}}$. Remark that $Q_{x_{0}}$ is a subsector of the sector, in $A_{x_{0}}$, with base point $x_{0}$ and ideal chamber $c_{+}$.  Now because $x \in A$ and  $\lim\limits_{n \to \infty} a^{n}(y) = \xi_{+}$, there exists $N_{x}>1$ such that for every $n>N_{x}$ we have $a^{n}(x) \in Q_{x_{0}}$. Therefore the geodesic segment from $x_{0}$ to $a^{n}(x)$ passes through the apartment $A$, for every $n> N_{x}$. Moreover, as $n$ increases, so does the length of the intersection $[x_{0},a^{n}(x)] \cap A$ which is contained in $Q_{x_{0}}$.

Assume next that $x_{0} \in A$. Permuting $x_{0}$ and $x$ and replacing $a$ by $a^{-1}$, we are reduced to the case that has already been treated.

\begin{figure}[t]
\begin{center}
\begin{tikzpicture}[yscale=1.7,xscale=1.6, vertex/.style={draw,fill,circle,inner sep=0.3mm}]{Fig.1}

\node[vertex]
(x) at (-0.50, 1.07) {};
\path (x) node[above] {$x_{0}$};

\node[vertex]
(x) at (0.70,-0.07) {};
\path (x) node[above] {$z$};

{
\draw[black!70!white, thin]
 (-0.50, 1.07)-- (0.70,-0.07);
}

 \filldraw[draw=black!700!white, fill=black!20!white, fill opacity = 0.3] (0:0.75cm) 
 (0.70,-0.07)--(2.70,0.70)--(5.50,0.07)--(3.50,-0.70)--(0.70,-0.07);

\filldraw[draw=red!10!white,fill=red!10!white, fill opacity = 0.5] (0:0.75cm) 
 (5.50,0.07)-- (6.30,1.40)--(6.90,0.80);

{
\draw[black!70!white, thin]
 (5.50,0.07)-- (6.30,1.40);
}

{
\draw[black!70!white, thin]
 (5.50,0.07)-- (6.90,0.80);
}

\node[vertex]
(x) at (6.50,1.00) {};
\path (x) node[above] {$h^{n}(x)$};

\node[]
(x) at (3.30,-0.40) {$Q_{x_{0}}\cap h^{n}(Q_{x})$};

{
\draw[red!70!white, thin]
 (-0.50, 1.07)-- (1.0,0.05);
}

{
\draw[red!70!white, thin]
 (1.0,0.05)--(8.00,-0.20);
}

\node[]
(x) at (8.15,-0.20) {$\xi_{+}$};

{
\draw[green!70!white, thin]
 (3.50,0.15)--(8.00,-0.02);
}

\node[vertex]
(x) at (5.50,0.07) {};
\path (x) node[above] {$h^{n}(y)$};

{
\draw[blue!70!white, thin]
 (-0.50, 1.07)--(1.228,0.12);
}

{
\draw[blue!70!white, thin]
 (1.228,0.12)--(5.10,0.30);
}

{
\draw[blue!70!white, thin]
 (5.10,0.30)--(6.50,1.00);
}

\end{tikzpicture}
\caption{\footnotesize The red line is the geodesic ray $[x_{0},\xi_{+})$; the green line is the translation of $y$ by powers of $a$; the blue line is the geodesic segment from $x_{0}$ to $a^{n}(x)$;}
\end{center}
\end{figure}
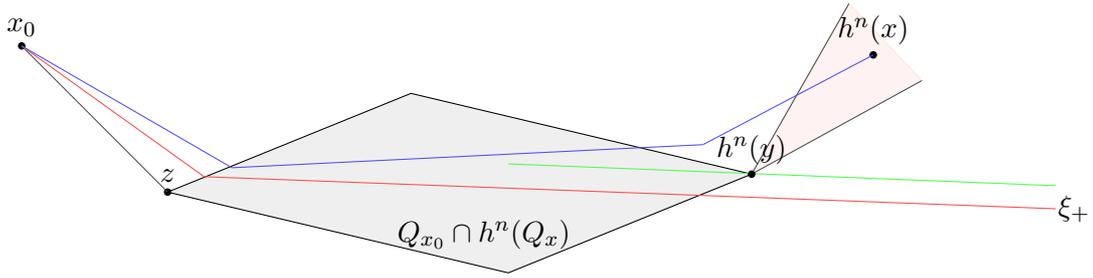

Consider now the remaining case of $x_{0}$ and $x \notin A$. Let $A_{x_{0}}$ and $A_{x}$ be two apartments in $\Delta$ such that  $x_{0} \in A_{x_{0}}$, $c_{+}$ is a chamber of $\partial A_{x_{0}}$, $x \in A_{x}$ and $c_{-}$  is a chamber of $\partial A_{x}$. Denote by $Q_{x_{0}}$ (resp., $Q_{x}$) a common sector of $A$ and $A_{x_{0}}$ (resp., of $A$ and $A_{x}$) corresponding to $c_{+}$ (resp., to $c_{-}$). As above, remark that $Q_{x_{0}}$ (resp., $Q_{x}$) is a subsector of the sector in $A_{x_{0}}$ (resp., $A_{x}$) with base point $x_{0}$ (resp., $x$) and ideal chamber $c_{+}$ (resp., $c_{-}$).
Let moreover $z$ be the base point of the sector $Q_{x_{0}}$. We can choose also $Q_{x}$ to be a sector, with base point $y$, such that $x$ is contained in the interior of the opposite sector to $Q_{x}$, with base point $y$, of the apartment $A_{x}$.

Since $\lim\limits_{n \to \infty} a^{n}(y)  = \xi_{+}$, there exists $N_{1}>1$ such that $a^{n}(y) \in Q_{x_{0}}$ and $z \in a^{n}(Q_{x})$, for every $n> N_{1}$. But every sector is convex. Therefore, so is any intersection of two sectors. Thus, for every $n> N_{1}$, the intersection $Q_{x_{0}} \cap a^{n}(Q_{x})$  is a non-empty convex subset of $A$ and so also of $a^{n}(A_{x})$ and $A_{x_{0}}$. As $n \to \infty$ the volume of $Q_{x_{0}} \cap a^{n}(Q_{x})$ grows strictly with $n$. Moreover, $Q_{x_{0}} \cap a^{n}(Q_{x})$ converges uniformly on compact sets to the sector $Q_{x_0}$.

For every $n > N_{1}$ we have $z \in a^{n}(Q_{x}) \subset  a^{n}(A_{x})$. Let $Q_{n}$ denote the sector in $a^{n}(A_{x})$ with base point $z$ and corresponding to the opposite ideal chamber of $c_{-}$ in $a^{n}(A_{x})$. This sector $Q_{n}$ contains as a subsector the sector with base point $a^{n}(y)$ and opposite to $a^{n}(Q_{x})$. Therefore $Q_{n}$ contains the vertex $a^{n}(x)$.

Let $n>N_{1}$ and consider the sector with base point $x_{0}$ and corresponding to the opposite ideal chamber of $c_{-}$ in $a^{n}(A_{x})$. This sector contains as a subsector the sector $Q_{n}$ and so also the vertex $a^{n}(x)$ and the intersection $Q_{x_{0}} \cap a^{n}(Q_{x})$. Since  $\lim\limits_{n \to \infty} a^{n}(x) = \xi_{+}$, there exists $N_{x}>N_{1}$ such that the geodesic segment from $x_{0}$ to $a^{n}(x)$ is passing through the convex set $Q_{x_{0}} \cap a^{n}(Q_{x}) \subset A$, for every $n > N_{x}$. Since $Q_{x_{0}} \cap a^{n}(Q_{x})$ tends uniformly to the sector $Q_{x_0}$, we infer that, given $T>0$, we may enlarge $N_{x}$ in such a way that for every $n \geq N_{x}$ the unique geodesic segment from $x_{0}$ to $a^{n}(x)$ contains a subsegment of length greater than $T$ entirely contained in $Q_{x_{0}} \cap a^{n}(Q_{x})$ and so in $A$. The proof is complete.
\end{proof}

\section{Euclidean buildings and strong transitivity}
\label{sect:Eucl-strong-trans}

Let $\Delta$ be a  building or a spherical building and $G$ be a group acting by automorphisms on $\Delta$. We say the $G$-action is  \textbf{strongly transitive}   if for any two pairs $(A_1,c_{1})$ and $(A_2,c_{2})$ consisting of an apartment $A_i$ and a chamber $c_i \in \Ch(A_i)$, there exists $g \in G$ such that $g(A_1)=A_2$ and $g(c_{1})=c_{2}$.

The goal of this section is to establish the equivalence between (i) and (ii) from Theorem~\ref{thm:main-thm}. We shall moreover provide several other characterizations in Proposition~\ref{transitivity2} below. Let us first notice that the implication from (i) to (ii) is well-known (see~\cite[Section~17.1]{Gar97}) and holds in full generality. In fact, in the first part of our discussion towards proving the converse implication, we leave the realm of locally compact groups and locally finite Euclidean buildings, and consider the broader framework of abstract groups acting on Euclidean buildings that are not supposed to be locally finite.

\subsection{A geometric `unipotent' radical}
\label{subsubsect:Geom-radical}

\begin{lemma}\label{lem:unip}
Let  $\Delta$ be a Euclidean building and $G \leq \Aut(\Delta)$ be any group of type-preserving automorphisms. Let $c \in \Ch(\partial \Delta)$ be a chamber at infinity. Then the set 
$$
G_c^0 :=\{g \in G_c \; | \;  g \text{ fixes some point of }\Delta\}
$$
 is a normal subgroup of the stabilizer $G_c =\{g \in G \; | \;  g(c)=c\}$. 
\end{lemma}

\begin{proof}
Let $g, h \in G_c^0$ and $x, y \in \Delta$ be points fixed by $g, h$ respectively. Since $g$ (resp., $h$) also fixes the chamber $c$ and is type-preserving, it must fix pointwise an entire sector $Q_x$ (resp., $Q_y$) emanating from $x$ (resp., $y$) and pointing to $c$. Any two sectors in a Euclidean building having the same chamber at infinity contain a common subsector. Thus $Q_x \cap Q_y$ is a non-empty subset of $\Delta$, which is clearly  fixed pointwise by $\langle g, h \rangle$. This implies that the product $gh$ belongs to $G_c^0$. Thus $G_c^0$ is a subgroup. It is clear that $G_c^0$ is invariant under conjugation by all elements of $G_c$. 
\end{proof}

There is another interpretation of the subgroup $G_c^0$ which will play an important role in our considerations. In order to describe it, we first need to recall the notion of retractions from a chamber at infinity. 

Given an apartment $A$ in a Euclidean building $\Delta$ and a chamber at infinity $c$ contained in $\partial A$, there is a map 
$$
\rho_{A, c} \colon \Delta \to A,
$$
called the \textbf{retraction} on $A$ based at $c$, characterised by the following properties: the restriction of $\rho_{A, c}$ to $A$ is the identity on $A$ and the restriction of  $\rho_{A, c}$ to any apartment $B$ whose boundary contains $c$, induces an isomorphism of $B$ onto $A$. We refer to  \cite[\S11.7 ]{AB} for more information. 

\begin{lemma}
\label{lem:BusemanRetraction}
Let  $\Delta$ be a Euclidean building and $G \leq \Aut(\Delta)$ be any group of type-preserving automorphisms. Let $c \in \Ch(\partial \Delta)$ be a chamber at infinity and $A$ be an apartment whose boundary contains $c$. Then for any $g \in G_c$, the map 
$$\beta_c(g) \colon A \to A : x \mapsto \rho_{A, c}(g(x))$$
is an automorphism of the apartment $A$, acting as a (possibly trivial) translation. Moreover the map
$$\beta_c \colon G_c \to \Aut(A)$$
is a group homomorphism whose kernel coincides with $G_c^0$.
\end{lemma}

\begin{proof}
Let $g \in G_c \setminus G_{c}^{0}$. Then $g(A)$ is an apartment whose boundary contains $c$. Therefore, the restriction of $\rho_{A, c}$ to $g(A)$ is an isomorphism from $g(A)$ to $A$ which fixes $c$. Hence $\beta_c(g)$ is indeed an automorphism of $A$ which fixes $c$. Since $G$ is type-preserving  $\beta_c(g)$ fixes all chambers in $\partial A$ and acts thus as a translation on $A$. 

Recall that the translation subgroup of $\Aut(A)$ acts freely on the set of special vertices, of the same type, of $A$. Therefore, the translation $\beta_c(g)$ is uniquely determined by its action on a given special vertex. In particular, in order to show that $\beta_c$ is a homomorphism, it suffices to find some special vertex $v$ of $A$ such that $\beta_c(gh)(v) = \beta_c(g)\beta_c(h)(v)$. Since $A$ and $h(A)$ have the chamber at infinity $c$ in common, they also contain a common sector associated to $c$. Now, for any vertex $v$ deep enough in that sector, we have $h(v) \in A$. Therefore $\rho_{A, c}(h(v)) = h(v)$ and hence 
$$
\begin{array}{rcl}
\beta_c(g) \beta_c(h)(v) & = & \beta_c(g)(\rho_{A, c}(h(v)))\\
& = &  \beta_c(g)(h(v))\\
& = & \rho_{A, c}(gh(v))\\
& =& \beta_c(gh)(v).
\end{array}
$$
This confirms that $\beta_c$ is indeed a homomorphism. 

The fact that $G^0_c \leq \Ker(\beta_c)$ is clear from the definition. Conversely, given any element  $g \in \Ker(\beta_c)$, the element $g$ must fix pointwise the intersection $A \cap g(A)$, and hence $g \in G^0_c$. Thus $G^0_c = \Ker(\beta_c)$. 
\end{proof}

\subsection{Trees in Euclidean buildings}

The following property of Euclidean buildings is well-known and useful. 

\begin{lemma}
\label{lem:TreeWalls}
Let $\Delta$ be a (thick, resp., locally finite) Euclidean building of dimension $n$. 
Let $\sigma, \sigma' \subset \partial \Delta$ be a pair of opposite panels at infinity. We denote by $P(\sigma, \sigma')$ the union of all apartments of $\Delta$ whose boundaries contain $\sigma$ and $\sigma'$. Then  $P(\sigma, \sigma')$ is a closed convex subset of $\Delta$, which splits canonically as a product
$$ P(\sigma, \sigma') \cong T \times \RR^{n-1},$$
where $T$ is a (thick, resp., locally finite) tree whose ends are canonically in one-to-one correspondence with the elements from the set $\Ch(\sigma)$ of all ideal chambers having $\sigma$ as a panel. Under this isomorphism, the walls of $\Delta$ contained in $ P(\sigma, \sigma')$ correspond to the subsets of the form $\{v\} \times \RR^{n-1}$ with $v$ a vertex of $T$.
\end{lemma}

\begin{proof}
See  \cite[Chapter 10, \S2]{Ron89}. 
\end{proof}

\subsection{Stabilizers of pairs of opposite panels}

The following result is also well-known; we provide a proof for the reader's convenience. 

\begin{lemma}\label{lem:OppPanels}
Let $Z$ be a thick spherical building and $G \leq \Aut(Z)$ be a strongly transitive group of type-preserving automorphisms. Then, for any pair of opposite panels $\sigma, \sigma'$, the stabilizer $G_{\sigma, \sigma'}$ is $2$-transitive on the set of chambers $\Ch(\sigma)$. 
\end{lemma}

\begin{proof}
Let $c \in \Ch(\sigma)$ and $x, y \in \Ch(\sigma)$ be two chambers different from $c$. Let $c' = \proj_{\sigma'}(c)$. Then $x$ and $y$ are both opposite $c'$. Therefore there is $g \in G_{c'}$ mapping $x$ to $y$. Since $G$ is type-preserving, it follows that $g$ fixes $\sigma'$ (because it fixes $c'$) and hence $\sigma$ (because it is the unique panel of $x$, respectively $y$, which is opposite $\sigma'$). Thus $g \in G_{\sigma, \sigma'}$. Moreover $g$ fixes $c'$, and hence also $c = \proj_\sigma(c')$. 

Thus, for any triple $c, x, y$ of distinct chambers in $\Ch(\sigma)$, we have found an element $g \in G_{\sigma, \sigma'}$ fixing $c$ and mapping $x$ to $y$. The $2$-transitivity of $G_{\sigma, \sigma'}$ on $\Ch(\sigma)$ follows. 
\end{proof}

\subsection{A criterion for strong transitivity}

\begin{lemma}\label{lem:ST:criterion}
Let  $\Delta$ be a thick Euclidean building and $G \leq \Aut(\Delta)$ be a group of type-preserving automorphisms. Then the following conditions are equivalent. 

\begin{enumerate}[(i)]
\item For each chamber  $c \in \Ch(\partial \Delta)$, the group $G_c^0$ is transitive on the set of chambers at infinity opposite $c$. 

\item For each chamber  $c \in \Ch(\partial \Delta)$ and each panel $\pi$ which is opposite to a panel of $c$, the group $G^0_{c, \pi} = G^0_c \cap G_\pi$ is transitive on the set $\Ch(\pi) \setminus \proj_\pi(c)$. 

\item $G$ is strongly transitive on $\Delta$. 
\end{enumerate}

\end{lemma}

\begin{proof}
(i) $\Rightarrow$ (ii): The hypothesis implies that $G$ is strongly transitive on the spherical building $\partial \Delta$. In particular $G$ is transitive on the set of apartments of $\Delta$.

Let $c \in \Ch(\partial \Delta)$, $\sigma$ be a panel of $c$ and  $\sigma' $ be a panel   which is opposite $\sigma$. Let also $x$ and $y$ be two chambers in  $\Ch(\sigma')$ that are both different from the projection $ \proj_{\sigma'}(c)$. In particular $x$ and $y$ are both opposite $c$. Therefore, there exists $g \in G^0_c$ mapping $x$ to $y$. Notice that $g$ fixes $\sigma$ since $G$ is type-preserving. Since $\sigma'$ is the unique panel of $x$ (resp., $y$) which is opposite $\sigma$, it follows that $g$ fixes  $\sigma'$ as well. Therefore $g \in G^0_{c, \sigma'}$. Thus (ii) holds.

\medskip
\noindent
(ii) $\Rightarrow$ (iii): We start with a basic observation. Let $A'$ and $A''$ be two apartments of $\Delta$ such that the intersection $A' \cap A''$ is a half-apartment. We claim that there is some $u \in G$ mapping $A'$ to $A''$ and fixing $A' \cap A''$ pointwise. 

Indeed, let $c \in \Ch(\partial (A' \cap A''))$ be an ideal chamber having a panel on the boundary of $\partial (A' \cap A'')$. 
The chamber $c$ has a unique opposite chamber $c'$ (resp., $c''$) in $\partial A'$ (resp., $\partial A''$). Notice that $c$ has some panel $\sigma$ on the boundary of the wall $A' \cap A''$; similarly the chambers $c'$ and $c''$ share a common panel $\sigma'$ which is opposite $\sigma$. By hypothesis, there is some $u \in G_{c, \sigma'}^0$ mapping $c'$ to $c''$. It follows that $u$ maps $A'$ to $A''$. Moreover, since $u$ fixes pointwise an entire sector of $A'$, it follows that $u$ fixes pointwise the intersection $A' \cap A''$, which is a half-apartment. The claim stands proven. 

\medskip 
The claim implies similar statements for half-apartments at infinity. It follows in turn that the stabilizer $G_c$ of every chamber at infinity is transitive on the set of chambers opposite $c$. In particular $G$ acts transitively on the collection of all apartments. Therefore, all it remains to show is that the stabilizer of each individual apartment acts transitively on the set of chambers of that apartment. 

Let thus  $A$ be an apartment of $\Delta$ and  $M$ be a wall of $A$. Let also $H$ and $H'$ be the two half-apartments of $A$ determined by $M$. Since $\Delta$ is thick, there exists a half-apartment $H''$ such that $H \cup H''$ and $H' \cup H''$ are both apartments of $\Delta$. By the claim above, we can find an element $u \in G$ fixing $H$ pointwise and mapping $H'$ to $H''$. Similarly, there are elements $v, w \in G$ fixing $H'$ pointwise and such that $v(H'') = H$ and $w(H) = u^{-1}(H')$. Now we set $r = vuw$. By construction $r$ fixes pointwise the wall $M$. Moreover we have
$$r(H) = vu u^{-1}(H') = v(H') = H'$$
and 
$$r(H') = vu(H') = v(H'') = H,$$
so that $r$ swaps $H$ and $H'$. It follows that $r$ stabilizes the apartment $A$ and acts on $A$ as the reflection through the wall $M$. Since this holds for an arbitrary wall of $A$, it follows that $\Stab_G(A)$ contains elements that realize every reflection of $A$. In particular $\Stab_G(A)$ is transitive on the chambers of $A$, as desired.

\medskip
\noindent
(iii) $\Rightarrow$ (i):
Let $x$ and $x'$ be two chambers opposite $c \in \Ch(\partial \Delta)$. Let moreover $A$ and $A'$ be the two apartments of $\Delta$ determined by the pairs $\{c, x\}$ and $\{c, x'\}$. They must share a common sector $s$ pointing to $c$. In particular they contain a common chamber $C \in \Ch(\Delta)$. By hypothesis there exists $g \in G$ mapping $A$ to $A'$ and fixing $C$. Since $G$ is type-preserving, it follows that $g$ fixes the intersection $A \cap A'$ pointwise. In particular $g$ fixes $s$ pointwise, and hence preserves $c$. Therefore $g \in G_c$. Since moreover $g$ fixes $C$ we have in fact $g \in G^0_c$, as desired. 
\end{proof}

Lemma~\ref{lem:ST:criterion} already implies that the desired equivalence holds in the special case of trees:

\begin{corollary}
\label{cor:trees}
Let $T$ be a thick tree and $G \leq \Aut(T)$ be an automorphism subgroup. If $G$ is $2$-transitive on $T(\infty)$, then $G$ is strongly transitive on $T$. 
\end{corollary}

\begin{proof}
The group $G$ must contain some hyperbolic isometry, otherwise it would fix a point in $T$ or $T(\infty)$, which is absurd. By the $2$-transitivity on the set of ends, it follows that $G$ contains hyperbolic isometries fixing any given pair of ends of $T$. Given $c \in T(\infty)$, let $\beta_c $ be the homomorphism from Lemma~\ref{lem:BusemanRetraction}. Remark that for the case of the tree one does not need the type-preserving condition in the hypothesis of Lemma~\ref{lem:BusemanRetraction}. Also remark that the image of the homomorphism $\beta_c $ is a cyclic group. Therefore the restriction of $\beta_c$ to the cyclic subgroup of $G_c$ generated by a hyperbolic isometry $t$ of minimal non-zero translation length is onto. Denoting the other fixed end of $t$ by $c'$, we deduce from Lemma~\ref{lem:BusemanRetraction}  that $G_c = G_{c, c'}.G_c^{0}$. Since $G_c$ is transitive on the ends of $T$ different from $c$, the same holds for  $G_c^0$. The desired conclusion now follows from Lemma~\ref{lem:ST:criterion}.
\end{proof}

\subsection{Apartments have cocompact stabilizers}

Let $\Delta$ be a thick Euclidean building and $G \leq \Aut(\Delta)$ be a group of type-preserving automorphisms. Remark that if $G$ is strongly transitive on $\partial \Delta$, then it is transitive on the set of apartments of $\Delta$. Therefore,  under the latter assumption $G$ is strongly transitive on $\Delta$ if and only if the stabilizer of any apartment is transitive on the set of chambers of that apartment. 
Our next goal is to establish the following approximation of this transitivity property. 

\begin{proposition}
\label{prop:CocompactApt}
Let $\Delta$ be a  thick Euclidean building and $G \leq \Aut(\Delta)$ be a group of type-preserving automorphisms. Assume that $G$ is strongly transitive on $\partial \Delta$. Then for any apartment $A$ of $\Delta$, the stabilizer $\Stab_G(A)$ acts cocompactly on $A$ (i.e., with finitely many orbits of chambers). 
\end{proposition}

\begin{proof}
In view of Corollary~\ref{cor:trees}, this holds clearly if $\Delta$ is a tree. We assume henceforth that $\Delta$ is higher-dimensional. 

\medskip
We start with a preliminary observation. Let $H, H'$ be two complementary half-apartments of $A$. We claim that there is some $g \in \Stab_G(A)$ which swaps $H$ and $H'$; in particular $g$ stabilizes the common boundary wall $\partial H = \partial H'$. 

In order to prove the claim, choose a pair of opposite panels $\sigma, \sigma'$ lying on the boundary at infinity of the wall $\partial H$. Notice that it makes sense to consider panels at infinity since $\Delta$ is not a tree, and thus the spherical building $\partial \Delta$ has positive rank. By Lemma~\ref{lem:OppPanels}, the stabilizer $G_{\sigma, \sigma'}$ acts $2$-transitively on $\Ch(\sigma)$. We now invoke Lemma~\ref{lem:TreeWalls} which provides a canonical isomorphism $P(\sigma, \sigma') \cong T \times \RR^{n-1}$, where $n = \dim(\Delta)$ and $T$ is a thick tree. The set $\Ch(\sigma)$ being in one-one correspondence with $T(\infty)$, we infer that $G_{\sigma, \sigma'}$ is $2$-transitive on $T(\infty)$, and hence strongly transitive on $T$ by Corollary~\ref{cor:trees}. Therefore, if $v$ (resp., $D$) denotes the vertex (resp., geodesic line) of $T$ corresponding to the wall $\partial H$ (resp., the apartment $A$), we can find some $g \in \Stab_G(A)$ stabilizing $D$ and acting on it as the symmetry through $v$. It follows that $g$ stabilizes $A$, the wall $\partial H  = \partial H'$ and swaps the two half-apartments $H$ and $H'$. This proves the claim. (We warn the reader that $g$ might however act non-trivially on the wall $\partial H$, so that it is not clear a priori that $g$ is the reflection through that wall.) 

\medskip
Let $V \subseteq A$ be a minimal Euclidean subspace invariant under $\Stab_G(A)$. Because of the above paragraph $\Stab_G(A)$ is not fixing any point in $A$. Therefore $V$ is not a point. Since $\Stab_G(A)$ acts on $A$ as a discrete group of Euclidean isometries, its action on $V$ is cocompact on $V$. Therefore, it suffices to show that $V=A$. Now, the boundary $V(\infty)$ is invariant under $\Stab_G(A)$. Since $G$ is strongly transitive on $\partial \Delta$, it follows that if $V(\infty) \neq \partial A$, then $V(\infty)$ is contained in some wall at infinity. Equivalently, $V$ is contained in some wall of $A$, say $M$. Let $M'$ be a wall parallel to, but distinct from $M$. By the claim above, we find an element $g \in \Stab_G(A)$ stabilizing $M'$ and swapping the half-apartments determined by it. In particular $g(M) \cap M $ is empty, which contradicts the fact that $g$ stabilizes $V$. This completes the proof. 
\end{proof}

\subsection{Geometric Levi decomposition}

A geometric analogue of the Levi decomposition of parabolic subgroups of semi-simple groups has been established in  \cite[Theorem~J]{CaMo} in the context of isometry groups of proper CAT(0) spaces. The following result can also be viewed as such a Levi decomposition. Note however that it concerns buildings that are not assumed to be locally finite, so it cannot be proved by invoking [loc. cit.].

\begin{lemma}\label{lem:Levi}
Let $\Delta$ be a   thick Euclidean building and $G \leq \Aut(\Delta)$ be a group of type-preserving automorphisms. Assume that $G$ is strongly transitive on $\partial \Delta$. 

Then for any pair $c, c'$ of opposite chambers  at infinity, the group $G_{c, c'}.G_c^0$ is of finite index in $G_c$. Moreover $G_c^0$ has finitely many orbits on the set of chambers opposite $c$, and each of these orbits is invariant under $G_{c, c'}$. 
\end{lemma}

\begin{proof}
Let $c, c'$ be a pair of opposite chambers  at infinity, and let $A$ be the unique apartment of $\Delta$ that they determine. 
 Since $\Stab_G(A)$ acts cocompactly on $A$ by Proposition~\ref{prop:CocompactApt}, so does the translation subgroup of $\Stab_G(A)$, which coincides with $G_{c, c'}$. 
 
Let now $\beta_c \colon G_c \to \Aut(A)$ be the homomorphism from Lemma~\ref{lem:BusemanRetraction}. Then for any $g \in G_{c, c'}$, we have $\beta_c(g) = g|_A$. We have just observed that $G_{c, c'}$ acts cocompactly on $A$, so that $\beta_c(G_{c, c'})$ is of finite index in $\Aut(A)$. In particular, the group $H = G^0_c.G_{c, c'} = \beta_c^{-1}(\beta_c(G_{c, c'}))$  is of finite index in $G_c$. In particular it acts with finitely many orbits on the set $\Opp(c)$ consisting of all chambers opposite $c$. 

We next claim that every $G^0_c$-orbit is invariant under $G_{c, c'}$. Indeed, given $d \in \Opp(c)$, there is some $x \in G_c$ with $d = x(c')$. Noticing that the quotient $G_c/G^0_c$ is abelian by Lemma~\ref{lem:BusemanRetraction}, and hence that $H$ is normal in $G_c$, we find 
$$\begin{array}{rcl}
G_{c, c'}(G^0_c(d)) & = & H(x(c')) \\
&= & x(H(c'))\\
& = & x(G^0_c.G_{c, c'}(c'))\\
& = & xG^0_c(c')\\
& = & G^0_c(x(c'))\\
&=& G^0_c(d).
\end{array}
$$
Thus the $ G^0_c$-orbit of $d$ is indeed left invariant by $G_{c, c'}$, as claimed. It follows that the $H$-orbits on $\Opp(c)$ coincide with the $G^0_c$-orbits, and hence, that there are only finitely many of these.
\end{proof}

In view of Lemma~\ref{lem:ST:criterion}, the final step in proving that $G$ is strongly transitive on $\Delta$ consists in showing that the index of $G_{c, c'}.G^0_c$ in $G_c$ is actually equal to one, since this means that $G^0_c$ is transitive on the chambers opposite $c$. It is only for this last step that we need to assume $\Delta$ to be locally finite and $G$ to be a closed subgroup of $\Delta$. The desired equality  $G_c = G_{c, c'}.G^0_c $ can then be deduce from the above results by invoking 
 \cite[Theorem~J]{CaMo}; alternatively, at this point it is also possible to provide a  direct argument. 

\begin{lemma}
\label{lem:Levi:decom}
Let $\Delta$ be a   thick Euclidean building and $G \leq \Aut(\Delta)$ be a group of type-preserving automorphisms. Assume that $G$ is strongly transitive on $\partial \Delta$.  Assume in addition that $\Delta$ is locally finite and that $G$ is closed. 

Then $G$ is strongly transitive on $\Delta$. 
\end{lemma}

\begin{proof}[First Proof of \ref{lem:Levi:decom} ]
We use the criterion from Lemma~\ref{lem:ST:criterion}. Let thus $c \in \Ch(\partial \Delta)$, let $\sigma$ be a panel of $c$ and $\sigma'$ be a panel opposite $\sigma$. 

Since $G$ is strongly transitive on $\partial \Delta$, we know from Lemma~\ref{lem:OppPanels} that $G_{\sigma, \sigma'}$ is $2$-transitive on $\Ch(\sigma)$. Moreover, Lemma~\ref{lem:Levi} implies that $G^0_{c, \sigma'}$ has finitely many orbits on $\Ch(\sigma') \setminus \proj_{\sigma'}(c)$, all of which are invariant under $G_{c, c'}$. 

Recall from Lemma~\ref{lem:TreeWalls} that $G_{\sigma, \sigma'}$ preserves a convex subset $P(\sigma, \sigma') \subset \Delta$ which is canonically isomorphic to a product of a tree $T$ with a flat factor $F$, and that the set of ends of $T$ is canonically in one-to-one correspondence with $\Ch(\sigma)$. We infer that $G^0_{c, \sigma'}$ fixes an end of $T$ and has finitely many orbits on the other ends. Since $G_{\sigma, \sigma'}$ is doubly transitive on the set of ends of $T$, it follows that some element of $G_{c, c'}$ must act as a hyperbolic isometry on $T$. 

Now we use the fact that $G^0_{c, \sigma'}$ is closed in $G$, and hence it acts properly on $P(\sigma, \sigma')$. Moreover, by definition, its induced action on the flat factor $F$ is trivial. It follows that $G^0_{c, \sigma'}$ acts properly on the tree $T$. Since $G^0_c .G_{c, c'}$ has finite index in $G_c$, one has that $G^0_{c, \sigma'} .G_{c, c'}$ is an open subgroup of $G_{c, \sigma'}$. From here we deduce that the $G^0_{c, \sigma'}$-orbit of the end of $T$ corresponding to $c'$ is open. On the other hand, any other orbit of the cyclic group generated by a hyperbolic element in $G_{c, c'}$ has the ends $c$ and $c'$ as limit points. This implies that $G^0_{c, \sigma'} $ is transitive on the set of ends of $T$ different from $c$. Equivalently, the group $G^0_{c, \sigma'} $ is transitive on $\Ch(\sigma') \setminus \proj_{\sigma'}(c)$. The desired conclusion follows from Lemma~\ref{lem:ST:criterion}.
\end{proof}

\begin{proof}[Alternative proof of Lemma \ref{lem:Levi:decom}]
Let $c \in \Ch(\partial \Delta)$, $c' \in \Opp(c)$ and denote by $A$ the unique apartment in $\Delta$ whose boundary contains $c$ and $c'$.

As $G$ is strongly transitive on $\partial \Delta$ we have that $G$ acts transitively on the set of all apartments of $\Delta$. Therefore by Proposition~\ref{prop:CocompactApt} we conclude that $G$ acts cocompactly on $\Delta$. Thus, by Theorem~\ref{thm:ExistenceStronglyReg}, $G$ contains a strongly regular hyperbolic element. Invoking the transitivity of $G$ on the set of apartments of $\Delta$ we conclude that every apartment in $\Delta$ admits a strongly regular hyperbolic element in $G$.

By the assumptions of the lemma we have that $G_{c}^{0}$ is an open subgroup of $G_{c}$. Thus every $G^0_{c} G_{c, c'}$-orbit in $\Opp(c)$ is closed in the cone topology on $\partial \Delta$. Applying Proposition~\ref{prop:DynamRegular} to a strongly regular element with unique translation apartment $A$ we obtain that $c'$ is an accumulation point for every $G^0_{c} G_{c, c'}$-orbit in $\Opp(c)$. By closedness we conclude that there is only one such orbit, thus $G_{c}=G^0_{c} G_{c, c'}$ and therefore $G^0_{c} $ is transitive on $\Opp(c)$. The desired conclusion follows from Lemma~\ref{lem:ST:criterion}.
\end{proof}

\begin{remark}\label{rem:LocallyInfinite}
We have already pointed out that the hypothesis that $\Delta$ to be locally finite and $G$ a closed subgroup of $\Aut(\Delta)$ was only used at the very last stage of the proof, namely in Lemma~\ref{lem:Levi:decom}. Without those extra hypotheses, one could also consider the subgroup $H$ of $G_{\sigma, \sigma'}$ defined by $H = \la G^0_{c, \sigma'} \; | \; c \in \Ch(\sigma) \ra$. By definition, this is a normal subgroup of $G_{\sigma, \sigma'}$ which acts trivially on the flat factor $F$ of $P(\sigma, \sigma')$. As in the proof above, we see that for each $c \in \Ch(\sigma) $ the group $G^0_{c, \sigma'}$ has finitely many orbits of ends on the tree $T$. Since $G_{\sigma, \sigma'}$ is doubly transitive on the set of ends $\partial T$, it follows that $H$ is transitive on $\bd T$ and has finitely many orbits on $\bd T \times \bd T$. This leads us to the following question: \textit{given a thick tree $T$ (not necessarily locally finite), a group $G \leq \Aut(T)$ acting doubly transitively on $\bd T$ and a normal subgroup $H $ of $G$ acting with finitely many orbits on $\bd T \times \bd T$, is it true that $H$ must itself be doubly transitive on $\bd T$?}

A positive answer to this question would imply that Lemma~\ref{lem:Levi:decom} holds without the assumptions on local finiteness of  $\Delta$ or closedness of $G$. In other words, the equivalence between (i) and (ii) from Theorem~\ref{thm:main-thm} would hold in full generality. 
\end{remark}

\subsection{Characterizing strong transitivity in the locally finite case}\label{sec:proof}

We finally record various characterizations of  strong transitivity for locally compact groups acting on  locally finite thick Euclidean buildings. 

\begin{theorem}
\label{transitivity2}
Let $G$ be a locally compact group acting continuously and  properly by type-preserving automorphisms on a locally finite thick Euclidean building $\Delta$.  Then the following are equivalent:
\begin{enumerate}[(i)]
\item
$G$ is strongly transitive on $\Delta$;

\item
$G$ is strongly transitive on $\partial \Delta$;

\item
$G$ has no fixed point in $\partial \Delta$ and for some chamber $c \in   \Ch(\partial \Delta)$, the stabilizer $G_c$ acts cocompactly on $\Delta$;
 
\item
$G$ acts cocompactly on $\Delta$, without a  fixed point in $\partial \Delta$, and there exists a compact open subgroup $K \leq  G$ having finitely many chamber-orbits on $\partial \Delta$.
\end{enumerate}
\end{theorem}

\begin{proof}
(i) $\Rightarrow$ (ii) is well-known, see \cite[Section~17.1]{Gar97}. 

\medskip \noindent 
(ii) $\Rightarrow$ (i)  is contained in Lemma~\ref{lem:Levi:decom}.

\medskip \noindent 
(i) $\Rightarrow$ (iii):
Since $G$ is strongly transitive on $\partial \Delta$, it is clear that $G$ cannot fix any point of $\partial \Delta$. Moreover, by strong transitivity, the stabilizer $G_c$ of a chamber at infinity $c$ is transitive on the chambers opposite $c$, and hence, also on the collection of apartments whose boundaries contain $c$. Now, any chamber $x$ in $\Delta$ is contained in such an apartment. In other words, if we fix an apartment $A$  whose boundary contains $c$, then $A$ contains a representative of every $G_c$-orbit of chambers.  Now our hypothesis ensures that $\Stab_G(A)$ acts cocompactly on $A$. This implies that the translation subgroup of $\Stab_G(A)$, which is of finite index, also acts cocompactly on $A$. This translation subgroup acts trivially on $\partial A$, and it is thus contained in $G_c$. It follows that $G_c$ acts cocompactly on $\Delta$. 

\medskip \noindent 
(iii) $\Rightarrow$ (ii):
By hypothesis $G$ acts cocompact on $\Delta$ and without a fixed point in $\partial \Delta$. Let $c \in \Ch(\partial A)$ be such that  $G_c$ acts cocompactly on $\Delta$. Then $G_c$ is transitive on the set of chambers opposite $c$ by \cite[Proposition 7.1]{CM09}. Since $G$ does not fix any point of $\partial \Delta$, it must contain an element $g$ mapping $c$ to a chamber $d$ sharing no simplex with $c$. Let now $A$ be an apartment whose boundary contains both $c$ and $d$. By the same arguments as in the proof of Lemma~\ref{lem:ST:criterion}, we see that every reflection associated to a wall separating $c$ from $d$ can be realised by an element of $\Stab_G(\partial A)$. The set of those reflections is easily seen to be a parabolic subgroup of the spherical Weyl group. Since $c$ and $d$ do not share any simplex, they are not contained in a common proper subresidue; therefore the parabolic subgroup in question must be the whole Weyl group. This implies that $\Stab_G(\partial A)$ is transitive on the chambers of $\partial A$. It follows that $G$ is transitive on $\Ch(\partial \Delta)$. Therefore, for all chambers $c' \in \Ch(\partial \Delta)$ the stabilizer $G_{c'}$ is transitive on the chambers opposite $c'$. From here we deduce that $G$ acts transitively on the set of all apartments in $\Delta$. Thus the strong transitivity of $G$ on $\partial \Delta$ follows. 

\medskip \noindent 
(i) $\Rightarrow$ (iv):
If $G$ is strongly transitive on $\Delta$ and $K$ denotes the stabilizer of a special vertex, then $K$ is transitive on the set of chambers at infinity: this is equivalent to the $KAK$-decomposition of $G$, whose existence follows from the Bruhat decomposition of $G$ with respect to the affine BN-pair. The desired implication follows. 

\medskip \noindent 
(iv) $\Rightarrow$ (iii):
 Let $K \leq G$ be a compact open subgroup having finitely many orbits of chambers at infinity and let $c \in \Ch(\partial \Delta)$. 

As the action of $G$ on $\partial \Delta$ is continuous (here we endow $\partial \Delta$ with the cone topology induced from $\Delta \cup \partial \Delta$), every $G$-orbit on $\Ch(\partial \Delta)$ is a finite union of $K$-orbits which by continuity are also compact in the cone topology (to see this use the finiteness of $K$-orbits). To prove the desired conclusion, it is enough to show that for every chamber at infinity $c \in \Ch(\partial \Delta )$, $G/G_{c}$ is compact in the quotient topology.

Take a class $hG_{c}$ in $G/G_{c}$. As $G (c) =K(g_{1}(c)) \sqcup ... \sqcup K(g_{m}(c))$, which is a finite union of $K$-orbits by hypothesis, there exists $i \in \left\lbrace 1,...,m \right\rbrace $ such that $h (c) \in K(g_{i}(c))$. So we have $h \in Kg_{i}G_{c}$ and thus $G/G_{c} \subset Kg_{1}G_{c} \sqcup ... \sqcup Kg_{m}G_{c}$. Moreover $G/G_{c}= \left( \cup_{i =1}^{m} Kg_{i} \right) G_{c}$. As $\cup_{i=1}^{m} Kg_{i}$ is compact in $G$, it follows that the projection from $G$ of $\cup_{i=1}^{m} Kg_{i}$ on $G/G_{c}$ is onto, thereby showing that $G/G_c$ is also compact. Thus the conclusion follows.
\end{proof}

\section{Gelfand pairs for groups acting on Euclidean buildings}
\label{sec:Gelfand-pair}

Let $G$ be a locally compact group and $K \leq G$ a compact subgroup. We denote by $C_{c}^{K}(G)$ the space of continuous, compactly supported   functions $\phi : G \to \CC$ that are  \textbf{$K$-bi-invariant}, i.e., functions that satisfy the equality
$\phi(kgk{'})=\phi(g)$  for every  $ g \in G $ and  all $ k, k{'} \in K$.
We view the $\CC$-vector space $C_{c}^{K}(G)$ as an algebra whose multiplication  is given by   the convolution product
\[
 \phi \ast \psi \colon x \mapsto  \int_{G} \phi(xg )\psi(g^{-1})d\mu(g).
\]
The pair $(G,K)$ is called a \textbf{Gelfand pair} if the convolution algebra $C_{c}^{K}(G)$ is commutative. A good introduction to the theory of Gelfand pairs may be consulted in \cite{vD09}.

The goal of this section is to prove the remaining implications from Theorem~\ref{thm:main-thm}. This will be achieved at the end of the section. Moreover, the technical core of the proof relies on the following lemma.

\begin{lemma}\label{lem:tech}
Let $G$ be a locally compact group acting continuously and properly by type-preserving automorphisms on a locally finite thick Euclidean building $\Delta$.  Suppose that the $G$-action is cocompact, without a fixed point in $\partial \Delta$ and that it is not  strongly transitive. 

Then, for any stabilizer $ K \leq G$ of a special vertex, there exist two strongly regular hyperbolic elements $\alpha, \beta \in G$ such that $\alpha \beta \not \in K\beta K \alpha K$. 
\end{lemma}

\begin{proof}
Let $x_0 \in \Delta$ be a special vertex and set $K: = G_{x_0}$. By Theorem~\ref{thm:ExistenceStronglyReg}, the group $G$ contains a strongly regular hyperbolic element $a$. We denote by $A$ its unique translation apartment. 

Furthermore, by Theorem~\ref{transitivity2}, the group $K$ has infinitely many orbits of chambers at infinity. Since $\partial A$ contains only finitely many chambers, we may thus find some chamber $c$ which does not belong to the $K$-orbit of any chamber in $\partial A$. 

We endow the set $\Delta \cup \partial \Delta$ with the cone topology (see~\cite[Chap.II.8]{BH99} for the definitions and properties), which turns $\Delta \cup \partial \Delta$ into a compact Hausdorff space. Thus $\Delta \cup \partial \Delta$ becomes a normal space as well. Moreover, each chamber at infinity is a compact subset of $\partial \Delta$; so is the visual boundary of every apartment. Since the $G$-action on  $\Delta \cup \partial \Delta$  is continuous and since $K$ is compact, we infer that $K \partial A$ is compact.  Moreover, since $K \partial A$ is a union of chambers,  all of which are different form $c$, it follows that the interior of $c$ is disjoint from $K \partial A$. 

Let $A'$ be an apartment containing $x_0$ and such that $c \in \Ch(\partial A')$. In view of Lemma~\ref{existance_reg_element}, we may thus find a strongly regular geodesic line $\rho \colon \RR \to \Delta$ such that $\rho(0) = x_0$, $\{\rho(nC)\}_{n \geq 0}$ is a set of special vertices of the same type, for some constant $C >0$, and $\rho(\infty) : = \lim\limits_{t \to \infty} \rho(t)$ lies  in the interior of $c$.

We consider now the standard open neighborhoods of $\rho(\infty)$, in the cone topology, which are of the form 
$$
U(\rho, r, \epsilon)=\{z \in \Delta \cup \partial \Delta \; | \; d(z, x_{0})>r \text{ and } d(p_{r}(z), \rho(r)) <\epsilon \},
$$ 
with $r,\epsilon >0$, and where $p_r(z)$ denotes the unique point on the geodesic from $x_0$ to $z$ at distance $r$ from $x_0$ (see~\cite[Chap.II.8]{BH99}).
Since $\rho(\infty)  \notin  K A \cup K\partial A$, which is a compact, thus closed set in the cone topology, there exist $r$ big enough and $\epsilon$ small enough such that for every $x \in KA \setminus \overline{B}(x_{0},r)$  we have $d(p_{r}(x), \rho(r)) \geq \epsilon$. Fix some $r$ and $\epsilon$ with this property. We have that $( K \partial A \cup K A \setminus \overline{B}(x_{0},r)) \cap U(\rho, r, \epsilon) = \varnothing$.

We now apply Proposition~\ref{prop:TechSRH} to the geodesic line $\rho$. This ensures that we may find a strongly regular hyperbolic element $b \in G$ admitting a translation axis which shares a segment of arbitrarily large length with the geodesic ray $\rho(\RR_+)$. Denoting by $\eta_+, \eta_-$ respectively the attracting and repelling endpoints of such a translation axis of  $b$, it follows that, upon replacing $b$ by $b^{-1}$, the geodesic ray $[x_0, \eta_+)$ shares a geodesic segment of arbitrarily large length with the geodesic ray $\rho(\RR_+)$, which does not necessarily pass through the vertex $x_0$. To conclude, we find a strongly regular hyperbolic element $b \in G$ such that $\eta_+ \in U(\rho, r, \epsilon)$. 

The elements $\alpha$ and $\beta$ that we are looking for will be defined as suitable high powers of $a$ and $b$ respectively. We now proceed to find those powers.  

Let $\rho' \colon \RR_+ \to \Delta$ be the geodesic ray such that $\rho'(0) = x_0$ and $\rho'(\infty) = \eta_+$. Since $\eta_+  \in U(\rho, r, \epsilon)$, we may find $r', \epsilon'>0$ such that
$$
U: = U(\rho', r', \epsilon') \subset U(\rho, r, \epsilon).
$$
Since for each $k \in K$ and $\xi \in \partial A$ we have that $k\xi \notin U(\rho, r, \epsilon)$, there exists $r_{k\xi}$ and $\epsilon_{k\xi}$ such that 
$$
U([x_0, k\xi), r_{k\xi}, \epsilon_{k\xi}) \cap U = \varnothing.
$$
The collection $\{U([x_0, k\xi), r_{k\xi}, \frac 1 2 \epsilon_{k\xi})\}_{k \in K, \xi \in \partial A}$ forms an open covering of the compact set $K \partial A$, from which we may extract a finite subcovering, corresponding to $k_1, \dots, k_n \in K$ and $\xi_1, \dots, \xi_n \in \partial A$. Set 
$$
r'' = \max \{r_{k_1\xi_1}, \dots, r_{k_n \xi_n}\}
\hspace{1cm} \text{and} \hspace{1cm}
\epsilon'' = \frac 1 2 \min \{\epsilon_{k_1\xi_1}, \dots, \epsilon_{k_n \xi_n}\}.
$$
Then by the triangle inequality, for each $k \in K$ and $\xi \in \partial A$, there is some $i \in \{1, \dots, n\}$ such that 
$$
U_{k\xi} = U([x_0, k\xi), r'', \epsilon'') \subset U([x_0, k_i\xi_i), r_{k_i\xi_i}, \epsilon_{k_i\xi_i}).
$$ 
In particular, $U_{k\xi}$ is an open neighborhood of $k\xi$ which is furthermore disjoint from $U$. Since $K$ fixes $x_0$, we observe that $U_{k\xi} = k U_\xi$ for all $\xi \in \partial A$. This implies that  $U \cap K U_{\xi} = \varnothing$
for all $\xi \in \partial A$.

By Proposition~\ref{prop:DynamRegular}, the sequence $\{a^m(\eta_+)\}_{m \geq 0}$ converges to some $\gamma \in \partial A$. In particular, for all  $m >0$ sufficiently large, we have $a^m(\eta_+) \in U_\gamma$. For such an $m$, we may thus find $r'''$ and $\epsilon'''$ with
$$ 
U' := U([x_{0},\eta_{+}), r''', \epsilon''') \subset U\cap a^{-m}(U_{\gamma}).
$$
We deduce that  $a^{m}(U') \subset U_{\gamma}$ and hence  $ka^{m}(U') \subset kU_{\gamma}=U_{k\gamma}$ for all $k \in K$. Since $U \cap K U_{\gamma} = \varnothing$ we have
\begin{equation}
\label{crucial}
ka^{m}(U') \cap U= \varnothing, \text{ for every }k \in K.
\end{equation}

Recall that $\eta_+$ is the attracting endpoint of a translation axis of $b$. Therefore, for each $y \in \Delta$, the sequence $\{b^n(y)\}_{n \geq 0}$ converges to $\eta_+$. In particular, for each finite set of points $F \subset \Delta$, we may find a sufficiently large $n$ such that $b^n(F) \subset U'$.  We apply this argument to the set $F = x_{0} \cup Ka^{m}(x_{0})$, which is indeed finite since $\Delta$ is locally finite. In this way we obtain a suitable large $n$.

Thus, for this $n$ found above we have: 
$$b^n (x_0) \in U' 
\hspace{1cm} \text{and} \hspace{1cm}
b^{n}k a^{m}(x_{0}) \in U' \subset U$$
for each $k \in K$. On the other hand, by~(\ref{crucial}), we also  have 
$$
k'a^{m}b^{n}(x_{0}) \notin U
$$
for all $k' \in K$. 
Therefore $K a^m b^n K \cap K b^n K a^m K =\varnothing$, since otherwise there would exist $k, k' \in K$ such that $b^{n} k a^{m} (x_0) = k' a^{m} b^{n}(x_0)$. 
Finally, setting $\alpha = a^m$ and $\beta = b^n$, we obtain $\alpha \beta \not \in K \beta K \alpha K$.
\end{proof}

We are now able to complete the proof of the main theorem. 

\begin{proof}[End of proof of Theorem~\ref{thm:main-thm}]
(i) $\Leftrightarrow$ (ii) is covered by Theorem~\ref{transitivity2}.

\medskip \noindent
(i) $\Rightarrow$ (iii) follows from \cite[Corollary~4.2.2]{Matsum}. 

\medskip \noindent
(iii) $\Rightarrow$ (i)  
Assume $G$ acts cocompactly on $\Delta$ and the $G$-action is \textit{not} strongly transitive. We need to prove that $(G, K)$ is \textit{not} a Gelfand pair, where $K: = G_{x_0}$ is the stabilizer of  a special vertex $x_0$. 

Assume first that $G$ fixes some point $\xi \in \partial \Delta$. It then follows from \cite[Theorem~M]{CaMo} that $G$ is not unimodular. Thus we are done in this case, since any locally compact group endowed with a Gelfand pair is unimodular by \cite[Proposition 6.1.2]{vD09}. 

We assume henceforth that $G$ does not fix any point in $\partial \Delta$. By Lemma~\ref{lem:tech}, we may find two strongly regular hyperbolic elements $\alpha, \beta \in G$ such that $\alpha \beta \not \in K \beta K \alpha K$. 
Set
$$\phi =\mathbf 1_{K \alpha K} \hspace{1cm} \text{and} \hspace{1cm}  \psi =\mathbf 1_{K \beta K}.
$$
Clearly the maps $\phi, \psi$ are both locally constant, hence continuous, and $K$-bi-invariant. Thus $\phi, \psi \in C_{c}^{K}(G)$ and we claim that $\phi * \psi \neq \psi * \phi$. 

Indeed, we first observe that
$$
\begin{array}{rcl}
\phi * \psi 
 (\alpha \beta) & = & \int_{G} \phi(\alpha \beta g)\psi(g^{-1})d\mu(g)\\
 & \geq &  \int_{\beta^{-1}K} \phi(\alpha \beta g)\psi(g^{-1})d\mu(g)\\
 & = & \int_{K} \phi(\alpha k)\psi(k^{-1}\beta)d\mu(k)\\
 &= & \mu(K)\\
 &> &0,
 \end{array}
$$
where $\mu$ is a  Haar measure on $G$.

On the other hand, consider
$$
\psi * \phi (\alpha \beta) =  \int_{G} \psi(\alpha \beta g)\phi(g^{-1})d\mu(g).
$$
For the integrand to be nonzero, we need that the variable $g$ satisfies both $g^{-1} \in K \alpha K$ and $\alpha \beta g \in K \beta K$. For such an element $g \in G$, we may find $k_1, k_2, k_3, k_4 \in K$ such that
$g^{-1} = k_1 \alpha k_2$ and $\alpha \beta g = k_3 \beta k_4$. This yields
$$
\alpha \beta = k_3 \beta k_4 g^{-1} =  k_3 \beta k_4 k_1 \alpha k_2 \in K \beta K \alpha K,
$$
which contradicts our choice of $\alpha$ and $ \beta $. Therefore, there is no choice of $g$ making the integrand nonzero in the integral defining $\psi * \phi (\alpha \beta)$. This implies that $\psi * \phi (\alpha \beta)=0$.
The claim stands proven, thereby confirming that the convolution algebra $C_{c}^{K}(G)$ is not commutative.
\end{proof}


\end{document}